\pdfoutput=1
\documentclass{amsart}
\usepackage{amsmath,amssymb,amsthm}

\makeatletter
\newtheorem*{rep@theorem}{\rep@title}
\newcommand{\newreptheorem}[2]{
\newenvironment{rep#1}[1]{
 \def\rep@title{#2 \ref{##1}}%
 \begin{rep@theorem}}%
 {\end{rep@theorem}}}
\makeatother

\newtheorem{thm}{Theorem}[section]
\newtheorem{lem}[thm]{Lemma}
\newtheorem{prop}[thm]{Proposition}
\newreptheorem{prop}{Proposition}
\newtheorem{cor}[thm]{Corollary}
\newtheorem{fact}[thm]{Fact}
\newtheorem{defn}{Definition}[section]

\newcommand{\kO}{\mathcal{O}}

\newcommand{\ot}{\operatorname{ot}}
\newcommand{\ls}{<_{\kO^*}}

\newcommand{\Nat}{\mathbb{N}}

\begin{document}

\title{Topological conjugations are not constructable}
\author{Linda Brown Westrick}
\email{westrick@math.berkeley.edu}

\begin{abstract}
We construct two computable topologically conjugate functions for
which no conjugacy is computable, or even hyperarithmetic, resolving
an open question of Kennedy and Stockman \cite{openproblems}.
\end{abstract}

\maketitle

\section*{Author's Note}
After this was written up I learned that most of the ideas below are
effectivizations of folk theorems which were known to descriptive set
theorists.  Hjorth has given a nice account of these in section 4.2 of
his book \emph{Classification and orbit equivalence
relations} \cite{hjorth}.

The non-logician who wants to understand the basic idea could still use
Sections \ref{Preliminaries} and \ref{Noncomputable} as an accessible expository construction of
a pair of computable topologically conjugate functions with no
computable conjugacy.

If one already has some familiarity with computable analysis and
is interested in the hyperarithmetic aspect, I would recommend reading
Hjorth's book instead, because the presentation there suggests a
cleaner way to encode linear orders into homeomorphisms.

\hfill  \emph{Linda Brown Westrick, June 6, 2013}

\section{Introduction}
Let $I$ denote the closed unit interval.  Two continuous functions
$f,g: I \rightarrow I$ are \emph{topologically conjugate} if there is
a homeomorphism $h$ of $I$ such that $f\circ h = h\circ g$.  The
function $h$ is called a \emph{topological conjugation} or a
\emph{conjugacy}.  In \cite[pg. 298]{openproblems} Ingram records the following
question of Kennedy and Stockman: Given $f$ and $g$ which are
topologically conjugate, how can one construct a conjugacy?  If
constructing a conjugacy $h$ means providing an algorithm which
computes arbitrarily good approximations to $h$, we show that in
general there is no such construction.
\begin{repprop}{noncomp}
  There are two computable topologically conjugate functions with no
  computable conjugacy.
\end{repprop}

It is possible to force conjugacies to be much less constructable than
merely noncomputable.  The hyperarithmetic functions, a superset of
the computable functions, include any function that can be
``constructed'' as the result of a transfinite computation of
countable ordinal length, where the order type of the computation
length must be computable as a linear order
 in the sense of Section \ref{2w}.  For an
introduction to the hyperarithmetic hierarchy we refer the reader to
\cite{sacks}.

\begin{repprop}{nonhyp}
There are two computable topologically conjugate functions with no
  hyperarithmetic conjugacy.
\end{repprop}

Both constructions rely centrally on the fact that any topological
conjugacy of $f$ and $g$ must include an order isomorphism from the
fixed points of $f$ to the fixed points of $g$.  We use pairs of
computable linear orderings without any computable order isomorphism to
specify the fixed points of $f$ and $g$.

In Section \ref{Preliminaries} we define the needed notions from
computability theory.  In Section \ref{Noncomputable} we construct two
computable topologically conjugate functions with no computable
conjugacy.  In Section \ref{Nonhyperarithmetic}, we build on the
methods of Section \ref{Noncomputable} to construct two computable
topologically conjugate functions with no hyperarithmetic conjugacy.

The author would like to thank Theodore Slaman and Antonio Montalb\'an
for useful conversations about this topic.

\section{Preliminaries}\label{Preliminaries}

We cover the notation and basic computability concepts for infinite
binary sequences in Section \ref{2w}, trees in Section \ref{trees},
and real-valued functions in Section \ref{real-valued functions}.
This section contains all the background needed for Sections
\ref{Noncomputable} and for the construction in Section \ref{construction}.

\subsection{Computability in Cantor space}\label{2w}
\emph{Cantor space}, denoted $2^\omega$, is the set of all infinite sequences
of 0's and 1's.  If $X\in 2^\omega$, then $X(n)$ refers to the $n$th
element of $X$.  For $X,Y\in 2^\omega$, we say that $X<Y$ if $X\neq Y$
and $X(n) < Y(n)$ where $n$ is the first place where they differ.
Elements of Cantor space are also identified in the natural way with
subsets of the natural numbers.

A set $X\in 2^\omega$ is \emph{computable} if there is an algorithm
which, on input $n$, outputs $X(n)$.  Formally, algorithms are
represented as \emph{Turing machines}.  A Turing machine accepts
natural numbers as inputs.  On a given input, a Turing machine may
output a natural number, or it may run forever.  If $\Gamma$ is a
Turing machine, then $\Gamma(n)$ denotes its output on input $n$, if this output
exists.  For more details about Turing machines we refer the reader to
\cite{soare}.

A function $f:\mathbb{N}\rightarrow\mathbb{N}$ is called computable if
there is a Turing machine $\Gamma$ such that $\Gamma(n) = f(n)$ for
all $n$.

A set $X$ is called \emph{computably enumerable} if $X$
is empty or if there is a Turing machine $\Gamma$ which halts on all
its inputs such that $X = \{ \Gamma(n) : n \in \mathbb{N}\}$.  In this
case the sequence $\Gamma(0), \Gamma(1), \Gamma(2), \dots$ is called
an \emph{enumeration} of $X$.  We will use the following fact:

\begin{fact} There is a computably enumerable set $A$ which is not computable.
\end{fact}

An \emph{oracle Turing machine} is a Turing machine which is permitted
to access arbitrary bits of a set called the \emph{oracle} as a part
of its computation.  The oracle is an element of $2^\omega$.  The
output of an Turing machine $\Gamma$ with oracle $X$ on input $n$ is
denoted $\Gamma^X(n)$.

If $X,Y\in 2^\omega$, we say $X$ \emph{computes} $Y$ if there is an
oracle Turing machine $\Gamma$ such that $\Gamma^X(n) = Y(n)$ for all
$n$.  If $X$ computes $Y$ and $X$ is computable, then $Y$ is also
computable.

A function $f:2^\omega \rightarrow 2^\omega$ is called computable if
there is an oracle Turing machine $\Gamma$ such that for all $X$ and
$n$, $\Gamma^X(n) = f(X)(n)$.  In this case, we say that $f(X)$ is
\emph{uniformly computable} from $X$.

There is a bijective \emph{pairing function} $\langle \cdot,\cdot
\rangle: \mathbb{N}\times\mathbb{N} \rightarrow \mathbb{N}$, whose
inverse is computable in the sense that the maps $\langle n, m \rangle
\mapsto n$ and $\langle n, m \rangle\mapsto m$ are computable.  This
function is useful for encoding information into subsets of
$\mathbb{N}$.  It is also useful for combining the information from
multiple $X\in 2^\omega$ in an orderly way.  Given
$\{X_n\}_{n\in\mathbb{N}}$ such that each $X_n\in 2^\omega$, we write
$\bigoplus_n X_n$ to denote the set $Y\in 2^\omega$ such that
$Y(\langle n, m\rangle) = 1$ if and only if $m \in X_n$.

Functions on natural numbers can be encoded into subsets of $\Nat$ as
follows.  The code for a function $f:\mathbb{N}\rightarrow\mathbb{N}$
is the set \{$\langle n, f(n) \rangle : n\in\Nat\}$.  This $f$ is
computable (computability for functions from $\Nat$ to $\Nat$ is
defined above) if and only if its code is a computable subset of
$\Nat$.

Relations on natural numbers can also be encoded into subsets of
$\Nat$.  In Section \ref{Nonhyperarithmetic} we consider linear orders
on subsets of the natural numbers.  A linear order $(A,<_R)$ can be
encoded as $\{\langle a,b \rangle : a,b\in A \text{ and } a<_R b\}$.
A linear order is \emph{computable} if and only if its code is
computable.

For more information about computability theory, we refer the reader
to \cite{soare}.

\subsection{Trees and Finite Binary Strings}\label{trees}
We write $2^{<\omega}$ for the set of all finite binary strings.  A
tree is a subset of $2^{<\omega}$ which is closed under taking initial
segments.  
The empty string is denoted $\emptyset$.  If $\sigma\in2^{<\omega}$,
we write $|\sigma|$ for the length of $\sigma$, and $\sigma(n)$ for
the $n$th element of $\sigma$, where $n$ starts at $0$.  If $\sigma,\tau \in 2^{<\omega}$, we
write $\sigma \subseteq \tau$ if $|\sigma|\leq |\tau|$ and for each
$n<|\sigma|$, $\sigma(n) = \tau(n)$.  We write
$\sigma^\smallfrown\tau$ to denote the concatenation of $\sigma$ and
$\tau$.  By $\sigma^\smallfrown\overline{1}$ (respectively
$\sigma^\smallfrown\overline{0}$) we mean the real $X\in 2^\omega$
such that $X(n) = \sigma(n)$ if $n<|\sigma|$ and $X(n) = 1$
(respectively $X(n)=0$) otherwise.

Trees can be encoded into subsets of $\Nat$ as follows.  There is a bijective
correspondence between $2^{<\omega}$ and $\Nat$ which causes the basic
functions and relations on finite strings defined above to be
computable.  The code for a tree is the subset of $\Nat$ consisting of
the codes for each of its finite strings.  A tree is called
\emph{computable} if its code is computable.

If $X\in 2^\omega$, $X\upharpoonright n$ denotes the finite binary
string which is the first $n$ bits of $X$.  If $T$ is a tree, then
$[T]\subseteq 2^\omega$ is the set of all $X$ such that
$X\upharpoonright n \in T$ for all $n$.  The set $[T]$ is also called
the path set of $T$ and its elements are called paths.

\subsection{Computability for Reals and Real-Valued Functions}\label{real-valued functions}

A real number in $[0,1]$ is computable if and only if its binary
decimal expansion is computable.  However, we do not use binary
expansions to represent real numbers because these expansions behave
badly near dyadic rationals.  (Knowing that $x \in (.5-\varepsilon,
.5+\varepsilon)$ gives us no information about the first digit of
$x$'s binary decimal expansion, no matter how small $\varepsilon$ is.)

The rational numbers can be encoded into the natural numbers as
follows.  A code for a rational number $q$ is a natural number
$m=\langle s, \langle a, b \rangle \rangle$ such that $q =
(-1)^s\frac{a}{b}$.

A code for a real number $r$ is meant to encode a Cauchy sequence of
rationals $\{q_n\}_{n\in\mathbb{N}}$ converging to $r$.  Formally, a
\emph{code} for $r$ is a set $X \in 2^\omega$ such that 
\begin{enumerate}
\item for each
$n\in\mathbb{N}$, there is exactly one $m$ such that $\langle n, m
\rangle \in X$
\item this $m$ is a code for a rational number
$q$ such that $|q-r|<2^{-n}$.
\end{enumerate} Note that there are many different codes for each
real.  A real is \emph{computable} if it has a computable code.  This
means that a real $r$ is computable if and only if there is an
algorithm which, on input $n$, returns a rational approximation to $r$
that is accurate to within $2^{-n}$.

A code for a continuous function $f:I\rightarrow I$ is meant to
specify $f$ on a dense subset of $I$ and provide a modulus of uniform
continuity.  Let $q_1,q_2,\dots$ be a computable enumeration of codes
for all the rational numbers in $I$.  A code for a continuous function
$f:I \rightarrow I$ is a set $X = \bigoplus_n X_n$ such that
\begin{enumerate}
\item for each $i>0$, $X_i$ is a code for $f(q_i)$
\item $X_0$ is a code for a function $d:\Nat\rightarrow\Nat$ such
  that $|f(x)-f(y)|<2^{-m}$ whenever $|x-y|<2^{-d(m)}$.
\end{enumerate}
A continuous function $f:I \rightarrow I$ is \emph{computable} if and only
if it has a computable code.  Finite sums and products of computable
functions are computable.

For more information about computable analysis we refer the reader to
\cite{pour-el}.

\section{Noncomputable Conjugation}\label{Noncomputable}

In this section we construct two computable functions
$f,g:[0,1]\rightarrow[0,1]$ such that $f$ and $g$ are topologically
conjugate but not computably topologically conjugate.  Both $f$ and
$g$ are homeomorphisms of $[0,1]$ satisfying $f(x)\geq x$ and
$g(x)\geq x$ for all $x$.  Each of them has a set of fixed points
corresponding to the set of paths through a computable tree.  The
trees' path sets are order isomorphic, but have no computable order
isomorphism.  Since any conjugacy induces an order isomorphism, no
conjugacy can be computable.

First we specify the locations where it is possible for our functions
to have a fixed point.  The following Cantor set differs from the
canonical one in that it has been shrunk to fit in the interval
$[\frac{1}{3}, \frac{2}{3}]$.
\begin{defn}
Let $\mathcal{C}\subset[0,1]$ denote the set of $x$ whose ternary decimal expansion begins with $1$ and continues using only $0$ and $2$.  If $X\in 2^\omega$, let $c(X)$ denote the unique $x\in\mathcal{C}$ whose first ternary digit is $1$ and whose $n+1$st ternary digit is $2X(n)$.
\end{defn}

Observe that $c$ is an order isomorphism between $\mathcal{C}$ and $2^\omega$.  Also, considering $c$ as a function from $2^\omega$ to $\{X\in 2^\omega : X \text{ is a code for some } x \in \mathcal{C}\}$, both $c$ and $c^{-1}$ are computable.

\begin{lem}\label{functionconstruction}
  From any (code for a) computable tree $T\subseteq 2^{<\omega}$, one
  may uniformly compute a (code for a) computable
  $f_T:[0,1]\rightarrow [0,1]$ satisfying
\begin{enumerate}
\item $f_T$ continuous and increasing with $f_T(x) \geq x$ for all $x$
\item $f_T(x) = x$ if and only if $x \in \{0,1\}$ or $x\in\mathcal{C}$
  and $c^{-1}(x) \in [T]$.
\end{enumerate}
\end{lem}
\begin{proof}
  Let $b:[0,1]\rightarrow [0,1]$ be a smooth bump function which
  satisfies $b(0)=b(1)=b'(0)=b'(1)=0$ and for all $x$, $b'(x)<1$.  For
  example, we could have $b(x) = \frac{1}{K}e^{-\frac{1}{x(1-x)}}$ where
  $K$ is a constant chosen large enough to guarantee that for all $x$,
  $b'(x)<1$.

  Given $T$, we build a sequence of functions $h_n$ as follows.  Let
  $C_n$ denote the set of all $\sigma \in T$ such that $|\sigma|=n$.
  To each such $\sigma$, associate the closed interval $I_\sigma =
  [c(\sigma^\smallfrown \overline{0}),
  c(\sigma^\smallfrown\overline{1})]$.  Let $U_n = (0,1)\setminus
  \bigcup_{\sigma \in C_n} I_\sigma$.  Then $U_n$ is an open set
  consisting of finitely many connected components
  $(a_1,c_1),\dots,(a_r,c_r)$.  Let $$h_n(x) = \sum_{i=1}^r
  \frac{1}{3^{n}}b[a_i,c_i](x),$$ where $$b[a,c](x) =
  \left\{\begin{array}{ll} (c-a)b(\frac{x-a}{c-a}) & \text{ if
      } x \in [a,c] \\ 0 & \text{ otherwise,} \end{array}\right.$$
that is, $b[a,c]$ is the function $b$ scaled proportionally so that it is supported on $[a,c]$.  Then define $$f_T = x + \sum_{n=1}^\infty h_n(x).$$

The function $f_T$ is continuous because it is the uniform limit of
continuous functions.  Note also that each $h_n$ is differentiable
with $|h_n'(x)|<3^{-n}$ for each $n$ (the bumps that comprise $h_n$
have disjoint support).  Therefore, the sum $1+\sum_{n=1}^\infty
h_n'(x)$ also converges uniformly to a function $g$ satisfying
$\frac{1}{2}< g(x) < \frac{3}{2}$ for all $x$.  So $f_T$ is
differentiable with $f_T'(x) = g(x) > 0$ for all $x$, so $f_T$ is
increasing.  And because $h_n(x)\geq 0$ for all $x$, $f_T(x)\geq x$
for all $x$.  It is clear that $f_T(0) = 0$ and $f_T(1) = 1$.  And for
$x\in (0,1)$, $x\in \cup_n U_n$ if and only if $x\notin c([T])$, so
$f(x) = x$ for $x\in c([T])$ and $f(x)>x$ otherwise.

Finally, $f_T$ is computable because the following approximation
holds: $$\left|f_T(x) - \left(x + \sum_{n=1}^N h_n(x)\right)\right| <
\sum_{n=N+1}^\infty 3^{-n} = \frac{1}{2\cdot 3^N}.$$ This can be used
to compute $f_T(q)$ to any precision for any rational $q$.
Furthermore, $|f_T'(x)|<\frac{3}{2}$ for all $x$, so $|f_T(x)-f_T(y)|
< \frac{3}{2}|x-y|$, which gives a computable modulus of continuity.
\end{proof}

\begin{lem}\label{conjugacy}
 Let $P,Q\subseteq 2^{<\omega}$. The following are equivalent:
\begin{enumerate}
\item $f_P$ and $f_Q$ are topologically conjugate.
\item $\ot([P]) = \ot([Q])$
\end{enumerate}
Furthermore, if $h$ is a homeomorphism of $[0,1]$ such that $h\circ f_P
 = f_Q \circ h$, then $h$ computes an order isomorphism
$h^*:[P]\rightarrow[Q]$.
\end{lem}
\begin{proof}
First note that if $h\circ f_P = f_Q\circ h$, then $h$ is necessarily order preserving.  If $h$ were order reversing, consider any $x$ such that $f_P(x) > x$.  Then $f_P(x) = h^{-1}(f_Q(h(x)))> x$, so $f_Q(h(x))< h(x)$, which is impossible.

  Suppose $h\circ f_P=f_Q\circ h$ as above.  For $X\in[P]$, define
  $h^*(X) = c^{-1}(h(c(X)))$.  Since $c(X)$ is a fixed point of $f_P$,
 $h(c(X))$ is a fixed point of $f_Q$.  Since
  $c(X) \not\in \{0,1\}$, $h(c(X)) \not\in \{0,1\}$, so $h(c(X)) =
  c(Y)$ for some $Y\in[Q]$.  Therefore $h^*$ is well-defined.  By a similar argument, $h^*$ is onto. Because
  $c$ and $h$ are order preserving, so is $h^*$.  Because $c$ and
  $c^{-1}$ are computable, $h^*$ is computable from $h$.

  On the other hand, suppose $h^*:[P] \rightarrow [Q]$ is an order
  isomorphism.  Define $h$ as follows.  For $x \in c([P])$,
  the fixed points of $f_P$, define $h(x) = c(h^*(c^{-1}(x))$, the
  corresponding fixed point of $f_Q$.  Set $h(0)=0,h(1)=1$.  

  Now if $x$ is not a fixed point of $f_P$, $x$ lies in an interval
  $(a,b)$ such that $f_P(a) = a, f_P(b) = b$, and $f_P(z)>z$ for $z\in
  (a,b)$.  Given such an interval $(a,b)$, define $h\upharpoonright
  (a,b)$ as follows.  Fix some $x_0 \in (a,b)$, and some $y_0\in
  (h(a), h(b))$.  For all $z\in (a,b)$, $f_P(z)\in(a,b)$, since $f_P$
  is order preserving and $f_P(b) = b$.  Furthermore,
  $\lim_{n\rightarrow\infty} f_P^n(z) = b$, because this limit must be
  a fixed point of $f_P$ (apply $f_P$ to both sides).  Now since $f_P$
  is a strictly increasing function, its inverse is well-defined, and
  by similar arguments to the previous, $z>f_P^{-1}(z)>\dots
  >f_P^{-n}(z)>\dots$ with $\lim_{n\rightarrow\infty}f_P^{-n}(z) = a$.
  The analogous facts hold for the interval $(h(a), h(b))$ with
  respect to the function $f_Q$.  Let $i:[x_0, f_P(x_0))\rightarrow
    [y_0, f_Q(y_0))$ be any homeomorphism.  For $x\in(a,b)$,
      define $$h(x) = f_Q^n(i(f_P^{-n}(x)))$$ where $n\in\mathbb{Z}$
      is the unique integer such that $x \in [f_P^n(x_0),
        f_P^{n+1}(x_0))$.

  Observe that $h$ maps $[f_P^n(x_0),f_P^{n+1}(x_0))$ homeomorphically
  onto $[f_Q^n(y_0),f_Q^{n+1}(y_0))$.  By the arrangement of these
  intervals, $h$ is continuous at their endpoints, so $h$ maps $(a,b)$
  homeomorphically onto $(h(a), h(b))$.  In the same way we see that
  $h$ is continuous at the fixed points of $f_P$, and so
  $h:[0,1]\rightarrow[0,1]$ is a homeomorphism.

  We claim that $h\circ f_P = f_Q \circ h$.  For the fixed points of
  $f_P$ this is clear.  Given $x$ not a fixed point, let $(a,b)\ni x$
  and $x_0$ be as in the definition of $h$.  Then $$h(f_P(x)) =
  f_Q^n(i(f_p^{-n}(f_P(x)))) = f_Q( f_Q^{n-1}(i(f_P^{-(n-1)}(x)))) =
  f_Q(h(x))$$ where $n\in\mathbb{Z}$ is the unique integer such that
  $f_P(x) \in [f_P^n(x_0), f_P^{n+1}(x_0))$.
\end{proof}

\begin{lem}\label{pqnoncomp} There are two computable trees $P$ and $Q$ such that
  $\ot([P])=\ot([Q])$ but there is no computable $h^*:[P]\rightarrow
  [Q]$ witnessing the isomorphism.
\end{lem}
\begin{proof}
  Let $P = \{\underbrace{1\dots 1}_n\underbrace{0\dots 0}_m : n,m \in
  \omega\}$.  Then $\ot([P]) = \omega+1$.  Let $A$ be any computably
  enumerable set which is not computable.  Since $A$ is not
  computable, the complement of $A$ is infinite.  Let $a_0,a_1,\dots,a_n,\dots$ be an
  enumeration of the elements of $A$.  Let $Q = \{\underbrace{1\dots
    1}_n\underbrace{0\dots 0}_m : n,m\in \omega, n\neq a_i \text{ for
    any } i<m\}$.  Since $A$ is coinfinite, $\ot([Q]) = \omega +1$ as
  well.  Suppose $h^*:[P]\rightarrow [Q]$ is an order isomorphism.  Then $h^*(\overline{1}) = \overline{1}$ and
  $h^*({\underbrace{1\dots 1}_n}^\smallfrown\overline{0}) =
  {\underbrace{1\dots 1}_m}^\smallfrown\overline{0}$ where $m$ is the
  $n$th element in the complement of $A$.  Thus $h^*$ computes $A$ (to
  find the $n$th element of $\mathbb{N}\setminus A$, compute the bits of
  $h^*({\underbrace{1\dots 1}_n}^\smallfrown\overline{0})$ until the first
  0 appears).  So $h^*$ is not computable.
\end{proof}

\begin{prop}\label{noncomp}
  There exist two computable topologically conjugate functions with no computable conjugacy.
\end{prop}
\begin{proof}
  Let $P$ and $Q$ be as in Lemma \ref{pqnoncomp}.  By Lemma
  \ref{conjugacy}, $f_P$ and $f_Q$ are topologically conjugate, but
  any conjugation computes an order
  isomorphism between $[P]$ and $[Q]$, and is thus not
  computable.
\end{proof}

\section{Non-hyperarithmetic conjugation}\label{Nonhyperarithmetic}

We strengthen the result of the previous section by constructing two
topologically conjugate functions with no hyperarithmetic conjugacy.

No background beyond the material in Section \ref{Preliminaries} is
needed to understand most of the material in Section \ref{construction}, which contains the construction.  The exception is Lemma
\ref{pqnonhyp}, which assumes Corollary \ref{isomorders}, familiarity with the Turing jump, and the fact that if $X\in 2^\omega$ is hyperarithmetic, then anything computable from the jump of $X$ is also
hyperarithmetic.
Sections
\ref{hypnotions} and \ref{discussion} assume familiarity with the
hyperarithmetic hierarchy.  For an introduction to hyperarithmetic
theory, we refer the reader to \cite{sacks}.

\subsection{Isomorphic Computable Linear Orders With No Hyperarithmetic Isomorphism}\label{hypnotions}

By $\kO^*$ we mean the set of notations for recursive linear orderings
with no hyperarithmetic descending sequences, which is defined in \cite{sf} and further developed in \cite{harrison}.  It is defined as follows:
\begin{multline*}\kO^* = \cap \{ X : X \in \text{HYP} \wedge 1\in X \wedge (z\in X\rightarrow 2^z\in X) \wedge \\(\forall n[\phi_e(n)\in X \wedge \phi_e(n) \ls \phi_e(n+1)] \rightarrow 3\cdot 5^e \in X)\},\end{multline*}
where $\phi_e$ is the $e$th Turing functional and $\ls$ is the computably enumerable relation satisfying  $1 \ls x$ if $x\neq 1$; $z\ls 2^z$; $\phi_e(n) \ls 3\cdot 5^e$; and $a\ls b \wedge b \ls c \rightarrow a\ls c$.

Harrison \cite{harrison} proved the following structure theorem for $\kO^*\setminus\kO$, the nonstandard ordinal notations:
\begin{thm}[Harrison]
Suppose $a\in \kO^*\setminus \kO$.  Let $1+\eta$ be the order type of the rationals in $[0,1)$. Then there exists a unique $\alpha < \omega_1^{CK}$ such that $\{y : y \ls a\}$ has order type $\omega_1^{CK}\cdot(1 + \eta) + \alpha$.
\end{thm}

Our use of this structure theorem is limited to the following corollary:

\begin{cor}\label{isomorders}
  There exists a pair of isomorphic computable linear orderings such
  that no isomorphism between them is hyperarithmetic.
\end{cor}
\begin{proof}
Let $a\in \kO^* \setminus \kO$ with $\ot(\{ y : y\ls a\}) = \omega_1^{CK}\cdot(1+\eta)$.  Then there is
another nonstandard notation $b\ls a$ whose standard part is also $0$.  So $\ot(\{y : y\ls a \}) =\ot(\{y: y\ls b\})$.  For
contradiction, suppose that $i:\{y : y\ls a\}\rightarrow \{y: y\ls b\}$ is a hyperarithmetic
isomorphism.  Then $\{i^n(a)\}_n$ is a hyperarithmetic descending
sequence in $\{y : y\ls a\}$, contradicting that $a\in \kO^*$.
\end{proof}

\subsection{Construction of Topologically Conjugate Functions}\label{construction}
We encode arbitrary computable linear orders into the path sets of computable trees, guaranteeing the path sets are order isomorphic if and only if the orders were. 
\begin{lem}\label{ordertree}
  Uniformly in any linear ordering $R=(A,<_R)$, one may
  compute a tree $T_R\subseteq 2^{<\omega}$ and a labeling function
  $l_R:A\rightarrow T_R$ such that
\begin{enumerate}
\item For each $\sigma\in T$, the last bit of $\sigma$ is 0 if and only if $\sigma = l_R(a)$ for some $a\in A$.
\item For each $a
  \in A$, $l_R(a)^\smallfrown\overline{1}\in [T_R]$
\item For each $a,b\in A$, $a<_R b$ if and only if $l_R(a)^\smallfrown\overline{1}<l_R(b)^\smallfrown\overline{1}$.
\item\label{4} If $X\in [T_R]$, then $X = l_R(a)^\smallfrown\overline{1}$ for
  some $a\in A$ if and only if $X$ has a successor in $[T_R]$.
\item\label{5} The order type of $[T_R]$ depends only on the order type of $R$.
\item\label{6} Two linear orders $R$ and $S$ are isomorphic if and only if the associated $[T_R]$ and $[T_S]$ are isomorphic.
\end{enumerate}
\end{lem}
\begin{proof}
Construct $T_R$ and $l_R$ in stages as follows.  At stage $n$ we
decide which strings $\sigma$ of length $n$ belong to $T_R$, and define $l_R(a)$ for all elements $a$ which have been seen to be in $A$ by stage $n$.

 At stage $0$ put the empty string in $T_R$.  At stage $n+1$, for each
 $\tau\in T_R$ such that $|\tau| = n$, put $\tau^\smallfrown 1$ into
 $T_R$.  If no new element of $A$ has been enumerated in this time,
 the stage is completed.  Otherwise, there are two possibilities.  If
 the new element $a$ is $R$-bigger than any element that has been
 enumerated previously, put $\sigma = {\underbrace{1\dots
     1}_n}^\smallfrown 0$ in $T_R$ and let $l_R(a) = \sigma$.  On the
 other hand, if $b$ is least among the previously enumerated elements
 of $A$ such that $b>_Ra$, then already we have $\tau =
 l_R(b)^\smallfrown\underbrace{1\dots 1}_{n-|l_R(b)|} \in T_R$.  Put
 $\tau^\smallfrown 0$ in $T_R$ and define $l_R(a) = \tau^\smallfrown
 0$.  This completes stage $n+1$.  The tree $T_R$ and the function
 $l_R$ have now been defined.

The first three parts of the lemma follow directly from the construction.
For part \ref{4}, suppose that $X = l_R(a)^\smallfrown\overline{1}$.
Then $X$ has a successor in $[T_R]$: the leftmost path of $T_R$ that
begins with $(l_R(a)\upharpoonright (|l_R(a)|-1))^\smallfrown 1$.  On
the other hand, suppose that $X$ does not take this form.  Then either
$X = \overline{1}$, in which case it does not have a successor, or $X$
has infinitely many zeros, in which case it has infinitely many
labeled substrings $l_R(a_1)\subset l_R(a_2) \subset \dots \subset X$.
Then we have $l_R(a_1)^\smallfrown\overline{1} >
l_R(a_2)^\smallfrown\overline{1} > \dots$ with
$\lim_{n\rightarrow\infty}l_R(a_n)^\smallfrown\overline{1} = X$, so $X$ has no successor in
$[T_R]$.

The order type of $[T_R]$ can be described as follows.  Let
$\mathcal{U}$ be the set of all upward closed subsets of $A$ that have
no least element.  We define an ordering $<^*$ on $A\cup \mathcal{U}$
which extends $<_R$.  If $U\in \mathcal{U}$ and $a\in A$, say that $a
<^* U$ if and only if $a\not\in U$.  If $U,V\in
\mathcal{U}$, say that $U <^* V$ if and only if $U\setminus V$ is nonempty.  Then consider the order $\langle A \cup
\mathcal{U}, <^* \rangle$.  This new order has an order type which
depends only on the order type of $\langle A, <_R\rangle$.
Furthermore, this is the same order type that the tree $[T_R]$ has,
via the order preserving bijection $a\mapsto
l_R(a)^\smallfrown\overline{1}$, $U\mapsto \inf \left(\{\overline{1}
\} \cup \{l_R(a)^\smallfrown\overline{1} : a \in U\}\right)$.

Finally, an isomorphic copy of $R$ may be recovered from $[T_R]$ by restricting the domain of the latter to $\{X : X \text{ has a successor in } [T_R]\}$.  Thus if $S$ and $R$ are linear orders, $[T_R]$ and $[T_S]$ are isomorphic if and only if $R$ and $S$ are.
\end{proof}

Next we will see that from a isomorphism between path sets of such trees, an isomorphism between the original orders may be obtained (using one Turing jump).  Thus, the pair of isomorphic linear orders from Corollary \ref{isomorders} generates a pair of trees whose isomorphic path sets have no hyperarithmetic isomorphism.

\begin{lem}\label{pqnonhyp}
There exist computable trees $P,Q\subseteq 2^{<\omega}$ such that $\ot([P]) = \ot([Q])$ but there is no hyperarithmetic $h:[P]\rightarrow[Q]$ which
  witnesses the isomorphism.
\end{lem}
\begin{proof}
  Let $R= (A_R,<_R)$ and $S=(A_S,<_S)$ be two
  isomorphic computable linear orderings such that no isomorphism
  between them is hyperarithmetic.  Compute $l_R$, $l_S$, $T_R$ and $T_S$ as in Lemma
  \ref{ordertree}.  Let $h:[T_R]\rightarrow[T_S]$ be an
  order isomorphism.  We claim that there is an order isomorphism
  $h^*:A_R\rightarrow A_S$ which is computable in the jump of $h$.
By considering only those elements of $[T_R]$ and $[T_S]$ which have successors,
we see that the restriction $h:\{l_R(a)^\smallfrown\overline{1} : a \in A_R\} \rightarrow \{l_S(b)^\smallfrown\overline{1} : b \in A_S\}$ is an order isomorphism.    Given $a \in A_R$, compute $h^*(a)$ as follows.   Consider the
  enumeration, computable in $h$, of the bits of
  $h(l_R(a)^\smallfrown\overline{1})$.  Each time a
  zero appears in that enumeration, ask the $h$-jump oracle if there will be
  another zero.  Eventually the last zero will be found.  At that
  point $\sigma$ has been enumerated with $\sigma(|\sigma|) = 0$ and
  $\sigma = l_S(b)$ for some $b \in A_S$.  Return this $b$.  (Since $l_S$ is recursive,
  the search for $b$ such that $l_S(b) = \sigma$ is guaranteed to
  terminate.)  Because $a\mapsto
  l_R(a)^\smallfrown\overline{1}$, $h\upharpoonright
  \{l_R(a)^\smallfrown\overline{1} : a \in A_R\}$, and $b\mapsto
  l_S(b)^\smallfrown\overline{1}$ are each order isomorphisms, $h^*$ is an order isomorphism.  Since
  $h^*$ is not hyperarithmetic, $h$ is not
  hyperarithmetic either.
\end{proof}

\begin{prop}\label{nonhyp}
  There exist two computable topologically conjugate functions with no hyperarithmetic conjugacy.
\end{prop}
\begin{proof}
  Let $P$ and $Q$ be as in Lemma \ref{pqnonhyp}.  By Lemma
  \ref{conjugacy}, $f_P$ and $f_Q$ are topologically conjugate, but
  any conjugation computes an order
  isomorphism between $[P]$ and $[Q]$, and is thus not
  hyperarithmetic.
\end{proof}

\subsection{Discussion}\label{discussion}

The above construction reduces pairs of linear orders to pairs of
$C[0,1]$ functions such that the functions are topologically conjugate if and
only if the orders were isomorphic.  

\begin{cor}
The set of all (pairs of indices for) computable topologically conjugate pairs of functions is $\Sigma_1^1$-complete.
\end{cor}

\begin{proof}
Due to the continuity of all functions involved, $f$ and $g$ are
topologically conjugate if and only if \begin{equation}\label{sigma11}\exists h \left[h \text{ is a
    homeomorphism of } I \text{ and } (\forall x \in \mathbb{Q}\cap
  I)[f(h(x)) = h(g(x))]\right].\end{equation} Note that the matrix is arithmetic.  Therefore, the statement ``$f$ and $g$ are topologically conjugate'' is $\Sigma_1^1$.

On the other hand, it is known that the isomorphism problem for computable linear orders is $\Sigma_1^1$-complete.   Proofs are given in, for example, \cite[Theorem 4.4(d)]{gk} and \cite[Lemma 5.2]{cdh}.  We have demonstrated a computable reduction from the isomorphism problem for linear orders to the conjugacy problem for functions on the interval (Lemmas \ref{ordertree} and \ref{functionconstruction}).  Therefore the conjugacy problem is $\Sigma_1^1$-complete.
\end{proof}

Furthermore, by (\ref{sigma11}) the set of all $h$ such that $h$ is a
conjugation of $f$ and $g$ is $\Delta_1^1$ relative to $f$ and $g$.
Therefore, assuming $f$ and $g$ are computable (or even
hyperarithmetic) and topologically conjugate, they must have an
$\kO$-computable conjugacy.  In fact, the Gandy basis theorem
guarantees the existence of a hyperarithmetically low conjugacy.
Therefore, the result in Proposition \ref{nonhyp} is as strong as
possible.

\bibliographystyle{amsalpha}
\bibliography{tcbib}{}

\end{document}